\documentclass[10pt]{amsart}

\usepackage{amscd,amsmath,amsfonts,amssymb,amsthm,hyperref,geometry,graphicx}

\newtheorem{theorem}{Theorem}[section]

\newtheorem{corollary}[theorem]{Corollary}

\newtheorem{lemma}[theorem]{Lemma}

\newtheorem{definition}[theorem]{Definition}
\numberwithin{equation}{section}

\begin{document}
\title{Absolutely summing multilinear operators via interpolation}
\author[Albuquerque]{N. Albuquerque}
\address{Departamento de Matem\'atica, \\
\indent Universidade Federal da Para\'iba, \\
\indent 58.051-900 - Jo\~ao Pessoa, Brazil.}
\email{ngalbqrq@gmail.com}
\author[N\'u\~nez]{D. N\'u\~nez-Alarc\'on}
\address{Departamento de Matem\'atica, \\
\indent Universidade Federal de Pernambuco, \\
\indent 50.740-560 - Recife, Brazil.}
\email{danielnunezal@gmail.com}
\author[Santos]{J. Santos}
\address{Departamento de Matem\'atica, \\
\indent Universidade Federal da Para\'iba, \\
\indent 58.051-900 - Jo\~ao Pessoa, Brazil.}
\email{joedson@mat.ufpb.br }
\author[Serrano]{D. M. Serrano-Rodr\'iguez}
\address{Departamento de Matem\'atica, \\
\indent Universidade Federal de Pernambuco, \\
\indent 50.740-560 - Recife, Brazil.}
\email{dmserrano0@gmail.com}
\thanks{N. Albuquerque was supported by CAPES; J. Santos was supported by
CNPq (Edital Universal 14/2012)}

\begin{abstract}
We use an interpolative technique from \cite{abps} to introduce the notion
of multiple $N$-separately summing operators. Our approach extends and
unifies some recent results; for instance we recover the best known
estimates of the multilinear Bohnenblust-Hille constants due to F. Bayart,
D. Pellegrino and J. Seoane-Sep\'ulveda. More precisely, as a consequence of
our main result, for $1\leq t<2$ and $m>1$ we prove that 
\begin{equation*}
\left( \sum_{i_{1},\dots ,i_{m}=1}^{\infty }\left\vert U\left(
e_{i_{1}},\dots ,e_{i_{m}}\right) \right\vert ^{\frac{2tm}{2+(m-1)t}}\right)
^{\frac{2+(m-1)t}{2tm}} \leq \left[\prod_{j=2}^{m}\Gamma \left( 2-\frac{2-t}{%
jt-2t+2}\right) ^{\frac{t(j-2)+2}{2t-2jt}}\right] \left\Vert U\right\Vert
\end{equation*}
for all complex $m$-linear forms $U:c_{0}\times \cdot \cdot \cdot \times
c_{0}\rightarrow \mathbb{C}$.
\end{abstract}

\maketitle

\section{Introduction and preliminaries}

We use standard notations and notions from Banach space theory as, \emph{e.g.}, in \cite{djt}.
The Banach spaces $X_{1},...,X_{m},X,Y$ are considered over the scalar field 
$\mathbb{K}$, with $\mathbb{K}$ be $\mathbb{R}$ or $\mathbb{C}$. A
continuous linear operator between Banach spaces $u:X\rightarrow Y$ is
absolutely summing when $\left( \left\Vert u\left( x_{j}\right) \right\Vert
\right) _{j\in \mathbb{N}}\in \ell _{1}$ whenever $\left( x_{j}\right)
_{j\in \mathbb{N}}$ is unconditionally summable. The theory of absolutely
summing operators has its origins in the $50$'s with Grothendieck's resum%
\'{e} but only in 1966-67 that the class of summing operators was presented
in its modern form (see \cite{djt,MitiPelc,Pietsch} for more details).

The success of the linear theory of absolutely summing operators motivated
the emergence of a non linear theory. In 1983 A. Pietsch \cite%
{Pietsch_ideals} initiated a research program sketching the roots of the
multilinear theory. Now, the multilinear theory of absolutely summing
operators is a very fruitful field of nonlinear Functional Analysis with
important connections with other fields. We stress, for instance, the
striking advances in the estimates of the Bohnenblust--Hille constants and
its applications to the final solution of the optimal estimate of the Bohr
radius \cite{bps,BH} and in quantum information theory \cite{Mont}.

\vskip.1cm

Let $2\leq q<\infty$. A Banach space $X$ has \emph{cotype} $q$ if there is a
constant $C>0$ such that, no matter how we select finitely many vectors $%
x_{1},\dots,x_{n}\in X$, 
\begin{equation}
\left( \sum_{k=1}^{n}\Vert x_{k}\Vert^{q}\right) ^{\frac{1}{q}}\leq C\left(
\int_{I}\left\Vert \sum_{k=1}^{n}r_{k}(t)x_{k}\right\Vert ^{2}dt\right) ^{%
\frac{1}{2}},  \label{qqq}
\end{equation}
where $I:=[0,1]$ and $r_{k}$ denotes the $k$-th Rademacher function. The
smallest of all these constants is denoted by $C_{q}(X)$ and it is called
the cotype $q$ constant of $X$. In fact, up to the constant $C$ the
definition of cotype can be changed by replacing the $L_{2}$ norm by an $%
L_{p}$ norm in (\ref{qqq}). More precisely:

\begin{theorem}[Kahane Inequality]
Let $0<p,q<\infty $. Then there is a constant $\mathrm{K}_{p,q}>0$ for which 
\begin{equation*}
\left( \int_{I}\left\Vert \sum_{k=1}^{n}r_{k}(t)x_{k}\right\Vert
^{q}dt\right) ^{\frac{1}{q}}\leq \mathrm{K}_{p,q}\left( \int_{I}\left\Vert
\sum_{k=1}^{n}r_{k}(t)x_{k}\right\Vert ^{p}dt\right) ^{\frac{1}{p}},
\end{equation*}%
holds, regardless of the choice of a Banach space $X$ and of finitely many
vectors $x_{1},\dots ,x_{n}\in X$.
\end{theorem}

The previous theorem is a generalization of the Khinchine inequality, which
holds for $q=2$ and $X=\mathbb{K}$. In this case the optimal constants are
known and denoted by $\mathrm{A}_{p}^{\mathbb{K}}$. For real scalars, U.
Haagerup \cite{Ha} proved that 
\begin{equation}
\mathrm{A}_{p}^{\mathbb{R}} = \frac{1}{\sqrt{2}} \left(\frac{\Gamma \left( 
\frac{p+1}{2}\right)}{\sqrt{\pi }}\right) ^{-\frac{1}{p}},\ \ \text{ for }
1.85 \approx p_0 <p<2  \label{kinreal}
\end{equation}
and 
\begin{equation}
\mathrm{A}_{p}^{\mathbb{R}} = 2^{\frac{1}{p}-\frac{1}{2}},\ \ \text{ for }
1\leq p \leq p_{0}\approx 1.85.  \label{kinreal2}
\end{equation}
The exact definition of $p_{0}$ is the following: $p_0 \in (1,2)$ is the
unique real number satisfying 
\begin{equation*}
\Gamma \left( \frac{p_{0}+1}{2}\right) =\frac{\sqrt{\pi }}{2}.
\end{equation*}
For complex scalars, H. K\"{o}nig \cite{Konig} (see also \cite{kk}) using
Steinhaus variables instead of Rademacher functions has shown that 
\begin{equation}
\mathrm{A}_p^\mathbb{C} = \left( \Gamma \left( \frac{p+2}{2}\right)
\right)^{-\frac{1}{p}}\ \ \text{ for } 1 \leq p <2.  \label{kincomp}
\end{equation}

\vskip.3cm

The \emph{weak $\ell_{1}$-norm} of vectors $x_{1},\dots,x_{n}$ in a Banach
space $X$ is defined by 
\begin{equation*}
\Vert\left( x_{i}\right) _{i=1}^{n} \Vert_{w,1} := \sup_{\|x^{\prime
}\|_{X^{\prime}}\leq1} \sum_{i=1}^{n} \vert x^{\prime}(x_{i})\vert.
\end{equation*}

From now on $X,X_{1},\dots ,X_{m},Y$ will denote Banach spaces. By $\mathcal{%
L}\left( X_{1},\dots ,X_{m};Y\right) $ denote the Banach space of all
(bounded) $m$-linear operators $U:X_{1}\times \cdots \times X_{m}\rightarrow
Y$. For $1\leq r<\infty ,\,U\in \mathcal{L}\left( X_{1},\dots
,X_{m};Y\right) $ is called \emph{multiple $(r,1)$-summing}, if there exists
a constant $C>0$ such that 
\begin{equation*}
\left( \sum_{i_{1},\dots ,i_{m}=1}^{n}\left\Vert U\left(
x_{i_{1}}^{(1)},\dots ,x_{i_{m}}^{(m)}\right) \right\Vert _{Y}^{r}\right) ^{%
\frac{1}{r}}\leq C\prod_{k=1}^{m}\left\Vert \left( x_{i}^{(k)}\right)
_{i=1}^{n}\right\Vert _{w,1}
\end{equation*}%
for all finite choice of vectors $x_{i}^{(k)}\in X_{k},\,1\leq i\leq
n,\,1\leq k\leq m$. The vector space of all multiple $(r,1)$-summing
operators is denoted by $\Pi _{(r,1)}^{m}\left( X_{1},\dots ,X_{m};Y\right) $%
. The infimum, $\pi _{\left( r,1\right) }^{m}(U)$, taken over all possible
constants $C$ satisfying the previous inequality defines a complete norm in $%
\Pi _{\left( r,1\right) }^{m}\left( X_{1},\dots ,X_{m};Y\right) $.

We need to recall some useful multi-index notation: for a positive integer $%
n $ and a finite subset $D \subset \mathbb{N}$, we denote by $|D|$ the
cardinality of $D$ and define the index set 
\begin{equation*}
\mathcal{M}(D,n) := \left\{ \mathbf{i}=(i_k)_{k\in D} \in \mathbb{N}^{|D|}
;\ i_k \in\{1,\dots,n\} \ \text{for each} \ k\in D \right\}.
\end{equation*}
Futher, $\mathcal{P}_{k}(D)$ will denote the set of subsets of $D$ with
cardinality $k$, for $1\leq k\leq |D|$. When $D=\{1,\dots,m\}$, we will
simply write 
\begin{equation*}
\mathcal{M}(m,n) := \mathcal{M}\left(\{1,\dots,m\},n\right) = \left\{ 
\mathbf{i}=(i_{1},\dots,i_{m}) \in \mathbb{N}^m ;\ i_{1},\dots,i_{m}
\in\{1,\dots,n\}\right\}
\end{equation*}
and 
\begin{equation*}
\mathcal{P}_{k}(m) := \mathcal{P}_{k}\left(\{1,\dots,m\}\right).
\end{equation*}
For $S = \{ s_{1},\dots,s_{k}\} \in \mathcal{P}_{k}(m)$, its complement will
be $\widehat{S}:=\{1,\dots,m\} \setminus S$ and $\mathbf{i}_{S}$ shall mean $%
(i_{s_{1}},\dots,i_{s_{k}})\in \mathcal{M}(k,n)$.

The following well-known lemmata will be useful along this paper (we refer
to \cite[Lemma 2.2]{dps.coord} and \cite[Corollary 5.4.2]{garling}):

\begin{lemma}
\label{coty} Let $X$ be a cotype $q$ Banach space, $1\leq r\leq q$ and $%
\left( x_{\mathbf{i}}\right)_{\mathbf{i} \in \mathcal{M}(m,n)}$ be a matrix
in $X$. Then 
\begin{equation*}
\left( \sum_{\mathbf{i}\in \mathcal{M}(m,n)}\left\Vert x_{\mathbf{i}%
}\right\Vert _{X}^{q}\right) ^{\frac{1}{q}}\leq C_{q}(X)^{m}\mathrm{K}%
_{r,2}^{m}\left( \int_{I^{m}}\left\Vert \sum_{\mathbf{i}\in \mathcal{M}%
(m,n)}r_{\mathbf{i}}(t)x_{\mathbf{i}}\right\Vert ^{r}dt\right) ^{\frac{1}{r}}
\end{equation*}%
where, for each $\mathbf{i}=(i_{1},\dots ,i_{m})\in \mathcal{M}(m,n)$, $r_{%
\mathbf{i}}(t)=r_{i_{1}}(t_{1})\cdots r_{i_{m}}(t_{m})$ and $%
dt=dt_{1}...dt_{m}$.
\end{lemma}

\bigskip

\begin{lemma}
\label{mink.seq} For $0<p<q<+\infty$, and any sequence of scalars $%
\left(a_{ij}\right) _{i,j\in\mathbb{N}}$ we have 
\begin{equation*}
\left( \sum_{i}\left( \sum_{j}|a_{ij}|^{p}\right) ^{\frac{q}{p}}\right) ^{%
\frac{1}{q}}\leq\left( \sum_{j}\left( \sum_{i}|a_{ij}|^{q}\right) ^{\frac{p}{%
q}}\right) ^{\frac{1}{p}}.
\end{equation*}
\end{lemma}

\bigskip

\section{The interpolative approach}

We now recall the interpolative approach introduced in \cite{abps} that was
crucial (see \cite{bps}) to obtain the ultimate constants of the
Bohennblust-Hille inequalities and to provide the precise asymptotic growth
of the Bohr radius.

For $\mathbf{p}=(p_{1},\dots,p_{m})\in\lbrack1,+\infty)^{m}$, and a Banach
space $X$, we shall consider the space 
\begin{equation*}
\ell_{\mathbf{p}}(X) := \ell_{p_{1}}\left( \ell_{p_{2}}\left( \dots\left(
\ell_{p_{m}}(X)\right) \dots\right) \right),
\end{equation*}
namely, a vector matrix $\left( x_{\mathbf{i}}\right) _{\mathbf{i}\in%
\mathcal{M}(m,\mathbb{N})}\in\ell_{\mathbf{p}}(X)$ if, and only if, 
\begin{equation*}
\left( \sum_{i_{1}=1}^{\infty}\left( \sum_{i_{2}=1}^{\infty}\left(
\dots\left( \sum_{i_{m-1}=1}^{\infty}\left( \sum_{i_{m}=1}^{\infty
}\left\Vert x_{\mathbf{i}}\right\Vert_X^{p_{m}}\right) ^{\frac{p_{m-1}}{p_{m}%
}}\right) ^{\frac{p_{m-2}}{p_{m-1}}}\dots\right) ^{\frac{p_{2}}{p_{3}}%
}\right) ^{\frac{p_{1}}{p_{2}}}\right) ^{\frac{1}{p_{1}}}<+\infty.
\end{equation*}
When $X=\mathbb{K}$, we just write $\ell_{\mathbf{p}}$ instead of $\ell_{%
\mathbf{p}}(\mathbb{K})$. The core of the interpolative approach from \cite%
{abps} is summarized as follows (we sketch the proof for the sake of
completeness):

\begin{lemma}[Interpolation procedure]
\label{gen.interp} Let $m,n,N$ be positive integers and $\mathbf{q},\mathbf{q%
}(1),\ldots,\mathbf{q}(N)\in[1,+\infty)^m$ be such that $\left( \frac{1}{%
q_{1}},\dots,\frac{1}{q_{m}}\right)$ belongs to the convex hull of $\left( 
\frac{1}{q_{1}(k)}, \dots, \frac{1}{q_{m}(k)}\right) ,\,k=1,\dots,N$. Then
for all scalar matrix $\mathbf{a} = \left(a_{\mathbf{i}}\right)_{\mathbf{i}%
\in\mathcal{M}(m,n)}$, 
\begin{equation*}
\left\Vert \mathbf{a}\right\Vert _{\mathbf{q}}\leq\prod_{k=1}^{N}\left\Vert 
\mathbf{a}\right\Vert _{\mathbf{q}(k)}^{\theta_{k}},
\end{equation*}
\emph{i.e.}, 
\begin{equation*}
\left( \sum_{i_1=1}^n \left( \dots \left( \sum_{i_m=1}^n \vert a_\mathbf{i}
\vert^{q_m} \right)^\frac{q_{m-1}}{q_m} \dots \right)^\frac{q_1}{q_2}
\right)^\frac1{q_1} \leq \prod_{k=1}^N \left[ \left( \sum_{i_1=1}^n \left(
\dots \left( \sum_{i_m=1}^n \vert a_\mathbf{i} \vert^{q_m(k)} \right)^\frac{%
q_{m-1}(k)}{q_m(k)} \dots \right)^\frac{q_1(k)}{q_2(k)}\right)^%
\frac1{q_1(k)} \right]^{\theta_k},
\end{equation*}
where $\theta_{k}$ are the coordinates of $\left(\frac{1}{q_{1}},\dots,\frac{%
1}{q_{m}}\right)$ on the convex hull.
\end{lemma}

\begin{proof}[Sketch of the proof]
We just follows the lines of \cite[Proposition 2.1]{abps}. Proceeding by
induction on $N$ and using that, for any Banach space $X$ and $%
\theta\in\lbrack0,1]$, 
\begin{equation*}
\ell_{\mathbf{r}}(X)=\left[ \ell_{\mathbf{p}}(X),\ell_{\mathbf{q}}(X)\right]
_{\theta},
\end{equation*}
with $\frac{1}{r_{i}}=\frac{\theta}{p_{i}}+\frac{1-\theta}{q_{i}}$, for $%
i=1,\dots,m$. If 
\begin{equation*}
\frac{1}{q_{i}}=\frac{\theta_{1}}{q_{i}(1)}+\cdots+\frac{\theta_{N}}{q_{i}(N)%
},
\end{equation*}
with $\sum_{k=1}^{N}\theta_{k}=1$ and each $\theta_{k}\in\lbrack0,1]$, then
we also have 
\begin{equation*}
\frac{1}{q_{i}}=\frac{\theta_{1}}{q_{i}(1)}+\frac{1-\theta_{1}}{p_{i}},
\end{equation*}
setting 
\begin{equation*}
\frac{1}{p_{i}}=\frac{\alpha_{2}}{q_{i}(2)}+\cdots+\frac{\alpha_{N}}{q_{i}(N)%
},\ \ \ \mbox{and}\ \alpha_{j}=\frac{\theta_{j}}{1-\theta_{1}},
\end{equation*}
for $i=1,\dots,m$ and $j=2,\dots,N$. So $\alpha_{j}\in\lbrack0,1]$ and $%
\sum_{j=2}^{N}\alpha_{j}=1$. Therefore, by the induction hypothesis, we
conclude that 
\begin{equation*}
\left\Vert \mathbf{a}\right\Vert _{\mathbf{q}}\leq\left\Vert \mathbf{a}%
\right\Vert _{\mathbf{q}(1)}^{\theta_{1}}\cdot\left\Vert \mathbf{a}%
\right\Vert _{\mathbf{p}}^{1-\theta_{1}} \leq\left\Vert \mathbf{a}%
\right\Vert_{\mathbf{q}(1)}^{\theta_{1}} \cdot \left[ \prod_{j=2}^{N}\left%
\Vert \mathbf{a}\right\Vert _{\mathbf{q}(j)}^{\alpha_{j}}\right]
^{1-\theta_{1}} = \prod_{k=1}^{N}\left\Vert \mathbf{a}\right\Vert _{\mathbf{q%
}(k)}^{\theta_{k}}.
\end{equation*}
\end{proof}

Consequently, combining the previous result with Lemma \ref{mink.seq} the
following generalization of the Blei inequality arises (see \cite[Remark 2.2]%
{bps}):

\begin{lemma}[Bayart, Pellegrino, Seoane-Sepulveda]
\label{blei.interp} Let $m,n$ be positive integers, $1\leq k\leq m$ and $%
1\leq s\leq q$. Then for all scalar matrix $\left( a_{\mathbf{i}}\right) _{%
\mathbf{i}\in \mathcal{M}(m,n)}$, 
\begin{equation*}
\left( \sum_{\mathbf{i}\in \mathcal{M}(m,n)}\left\vert a_{\mathbf{i}%
}\right\vert^{\rho}\right)^{\frac{1}{\rho }} \leq \prod_{S\in \mathcal{P}%
_{k}\left( m\right)} \left( \sum_{\mathbf{i}_{S}} \left( \sum_{\mathbf{i}_{%
\widehat{S}}} \left\vert a_{\mathbf{i}}\right\vert^{q} \right)^{\frac{s}{q}}
\right)^{\frac{1}{s} \cdot \frac{1}{\binom{m}{k}}},
\end{equation*}
where 
\begin{equation*}
\rho :=\frac{msq}{kq+(m-k)s}.
\end{equation*}
\end{lemma}

\section{Multiple summing operators with multiple exponents}

In this section we apply the interpolation procedure to generalize results
of the theory of multiple summing multilinear operators. Our main result
recovers, with a new approach, one of the main results of \cite{dps.coord}.

\vskip.1cm

For Banach spaces $X_1,\dots,X_m$ and a proper non-void subset $D \subset
\{1,\dots,m\}$ let $X^D$ be the product $\prod_{k \in D} X_k$. A vector $x_D
\in X^D$ may be seen as an element $\widetilde{x_D} = \left(\widetilde{x_D^1}%
,\dots,\widetilde{x_D^m}\right) \in X_1 \times\dots\times X_m$, with $%
\widetilde{x^i_D}=x^i_D$, if $i\in D$, and $\widetilde{x^i_D}=0$, otherwise.
Given $U \in \mathcal{L} \left(X_1,\dots,X_m;Y\right)$, we define the map 
\begin{equation*}
\begin{array}{ccccccc}
U^D: & X^{\widehat{D}} & \rightarrow & \mathcal{L}\left(X^D;Y\right) &  &  & 
\\ 
& x_{\widehat{D}} & \mapsto & U^D_{x_{\widehat{D}}}: & X^D & \rightarrow & Y
\\ 
&  &  &  & y_D & \mapsto & U(\widetilde{x_{\widehat{D}}}+\widetilde{y_D}).%
\end{array}%
\end{equation*}
Clearly $U^D$ is well-defined and $\vert\widehat{D}\vert$-linear. Notice
that, for each $x_{\widehat{D}} \in X^{\widehat{D}},\, U^D_{x_{\widehat{D}}}$
is the restriction of $U$ to the $D$-coordinates, with the $\widehat{D}$%
-coordinates fixed through $x_{\widehat{D}}$. The following definition was
introduced in \cite{dps.coord}.

\begin{definition}
\label{one} Let $1\leq r<\infty$ and let $D $ be a proper subset of $%
\{1,\dots,m\}$. We say that $U \in \mathcal{L} \left(X_1,\dots,X_m;Y\right)$
is \emph{multiple $(r,1)$-summing in the coordinates of $D$} (or \emph{%
multiple $(r,1)$-summing in $D$}) whenever $U^D$ has its range in $%
\Pi_{(r,1)}^{\left\vert D\right\vert}\left(X^D;Y\right)$. Moreover, $U$ is 
\emph{separately $(r,1)$-summing} if $U$ is multiple $(r,1)$-summing in all
one point subset of $\{1,\dots,m\}$.
\end{definition}

\vskip.1cm

The following result came from a careful look at the argument in the proof
of \cite[Theorem 4.1]{dps.coord}. It provides estimates for bounded $m$%
-linear operators that are multiple $(r,1)$-summing in the coordinates of a
fixed index proper subset of $\{1,\ldots,m\}$.

\begin{theorem}[Defant, Popa, Schwarting]
\label{DPS} Let $Y$ be a cotype $q$ Banach space, $1\leq r\leq q$ and
suppose that $D \subseteq \{1,...,m\}$ is non-void and proper. If $U\in 
\mathcal{L}\left(X_1,\dots,X_m;Y\right)$ is multiple $(r,1)$-summing in the
coordinates of $D$, then 
\begin{equation*}
\left( \sum_{\mathbf{i}_{D}} \left( \sum_{\mathbf{i}_{\widehat{D}}}
\left\Vert U\left( x_{i_{1}}^{(1)},\dots ,x_{i_{m}}^{(m)}\right)
\right\Vert^{q} \right)^{\frac{r}{q}} \right)^{\frac{1}{r}} \leq \mathrm{A}%
_{q,r}^{|\widehat{D}|}(Y) \left\Vert U^{D}:X^{\widehat{D}} \rightarrow \Pi
_{(r,1)}^{|D|}\left( X^{D};Y\right) \right\Vert
\end{equation*}
for all finite choice of vectors $x_{1}^{(k)},\dots ,x_{N}^{(k)}\in X_{k}$,
with $\left\Vert \left( x_{j}^{(k)}\right) _{j=1}^{N}\right\Vert _{w,1}\leq
1 $, for $k=1,\dots ,m$ and $\mathrm{A}_{q,r}(Y):=C_{q}(Y)\mathrm{K}_{r,2}$.
\end{theorem}

Above and from now on, the symbol $\sum_{\mathbf{i}_{D}}$ means that we are
taking the sum over the indices $i_k$, with $k\in D$. Also the constant $%
\mathrm{A}_{q,r}(Y)$ is defined as above. The main result of this section
reads as follows:

\begin{theorem}
\label{4.2} Let $Y$ be a cotype $q$ Banach space, $1\leq r_{1},\dots
,r_{n}\leq q $ and $\{1,...,m\}$ be the disjoint union of non-void proper
subsets $C_{1},\dots ,C_{n}$. If $U\in \mathcal{L}\left( X_{1},\dots
,X_{m};Y\right) $ is multiple $(r_{k},1)$-summing in each coordinate subset $%
C_k$, for $k=1,\dots,n$, then 
\begin{align*}
& \left( \sum_{\mathbf{i}_{C_{1}}} \left( \sum_{\mathbf{i}_{C_{2}}} \left(
\dots \left( \sum_{\mathbf{i}_{C_{n}}} \left\Vert U
\left(x_{i_{1}}^{(1)},\dots,x_{i_{m}}^{(m)}\right) \right\Vert_Y^{q_n}
\right)^{\frac{q_{n-1}}{q_{n}}} \dots \right)^{\frac{q_{2}}{q_{3}}} \right)^{%
\frac{q_{1}}{q_{2}}} \right)^{\frac{1}{q_{1}}} \\
& \leq \prod_{k=1}^n \left[ \mathrm{A}_{q,r_k}^{|\widehat{C_k}|}(Y)
\left\Vert U^{C_k} : X^{\widehat{C_k}} \to \Pi_{r_k,1}^{|C_k|}
\left(X^{C_k};Y\right) \right\Vert \right]^{\theta_k},
\end{align*}
regardless of the finite choice of vectors $x_{1}^{(l)},\dots
,x_{N}^{(l)}\in X_{l}$ with $\left\Vert \left( x_{j}^{(l)}\right)
_{j=1}^{N}\right\Vert _{w,1}\leq 1$, $l=1,\dots ,m$. Here, each $q_{k}\in
\lbrack r_{k},q]$ is such that $\frac{1}{q_{k}}=\frac{\theta _{k}}{r_{k}}+%
\frac{\left( 1-\theta _{k}\right) }{q}$, with $\theta _{1},\dots ,\theta
_{n}\in \lbrack 0,1]$ and $\sum\limits_{k=1}^{n}\theta _{k}=1$.
\end{theorem}

\begin{proof}
Since $U$ is multiple $(r_{k},1)$-summing in each subset $C_{k}$, the
previous theorem assures that, for $x_{1}^{(l)},\dots ,x_{N}^{(l)}\in X_{l}$
with $\left\Vert \left( x_{j}^{(l)}\right) _{j=1}^{N}\right\Vert _{w,1}\leq
1 $, $l=1,\dots,m$, 
\begin{equation*}
\left( \sum_{\mathbf{i}_{C_{k}}}\left( \sum_{\mathbf{i}_{\widehat{C_{k}}%
}}\left\Vert U\left( x_{i_{1}}^{(1)},\dots,x_{i_{m}}^{(m)} \right)
\right\Vert_{Y}^{q}\right) ^{\frac{r_{k}}{q}}\right) ^{\frac{1}{r_{k}}}\leq 
\mathrm{A}_{q,r_{k}}^{\left\vert \widehat{C_{k}}\right\vert }(Y)\left\Vert
U^{C_{k}}:X^{\widehat{C_{k}}}\rightarrow \Pi _{(r_{k},1)}^{\left\vert
C_{k}\right\vert }\left( X^{C_{k}};Y\right) \right\Vert .
\end{equation*}%
for $k=1,\dots ,n$. Now, Lemma \ref{mink.seq} guarantees that we may change
the position of the exponents $r_{k}$ and $q$ (with the correspondent
indices): 
\begin{align*}
& \left( \sum_{\mathbf{i}_{C_{1}},\dots ,\mathbf{i}_{C_{k-1}}}\left( \sum_{%
\mathbf{i}_{C_{k}}}\left( \sum_{\mathbf{i}_{C_{k+1}},\dots ,\mathbf{i}%
_{C_{n}}}\left\Vert U\left( x_{i_{1}}^{(1)},\dots ,x_{i_{m}}^{(m)}\right)
\right\Vert _{Y}^{q}\right) ^{\frac{r_{k}}{q}}\right) ^{\frac{q}{r_{k}}%
}\right) ^{\frac{1}{q}} \\
&\leq \left( \sum_{\mathbf{i}_{C_{k}}}\left( \sum_{\mathbf{i}_{\widehat{C_{k}%
}}}\left\Vert U\left( x_{i_{1}}^{(1)},\dots ,x_{i_{m}}^{(m)}\right)
\right\Vert _{Y}^{q}\right) ^{\frac{r_{k}}{q}}\right) ^{\frac{1}{r_{k}}} \\
&\leq \mathrm{A}_{q,r_{k}}^{\left\vert \widehat{C_{k}}\right\vert
}(Y)\left\Vert U^{C_{k}}:X^{\widehat{C_{k}}}\rightarrow \Pi
_{(r_{k},1)}^{\left\vert C_{k}\right\vert }\left( X^{C_{k}};Y\right)
\right\Vert.
\end{align*}

On the other hand, the hypotheses on $q_{1},\dots ,q_{m}$ mean precisely
that $\left( \frac{1}{q_{1}},\dots ,\frac{1}{q_{m}}\right) $ belongs to the
convex hull of the points $\left( \frac{1}{q_{1}(k)},\dots ,\frac{1}{q_{m}(k)%
}\right) ,\ k=1,\dots ,n$, with 
\begin{equation*}
q_{j}(k):=%
\begin{cases}
r_{k},\text{ if }j\in C_{k}; \\ 
q,\text{ if }j\in \widehat{C_{k}},%
\end{cases}%
\end{equation*}%
for $k=1,\dots ,n$. Therefore, the interpolation method of Lemma \ref%
{gen.interp} gives us 
\begin{align*}
& \left( \sum_{\mathbf{i}_{C_{1}}}\left( \sum_{\mathbf{i}_{C_{2}}}\left(
\dots \left( \sum_{\mathbf{i}_{C_{n}}}\left\Vert U\left(
x_{i_{1}}^{(1)},\dots ,x_{i_{m}}^{(m)}\right) \right\Vert
_{Y}^{q_{n}}\right) ^{\frac{q_{n-1}}{q_{n}}}\dots \right) ^{\frac{q_{2}}{%
q_{3}}}\right) ^{\frac{q_{1}}{q_{2}}}\right) ^{\frac{1}{q_{1}}} \\
&\leq \prod_{k=1}^{n}\left[ \left( \sum_{\mathbf{i}_{C_{1}},\dots ,\mathbf{i}%
_{C_{k-1}}}\left( \sum_{\mathbf{i}_{C_{k}}}\left( \sum_{\mathbf{i}%
_{C_{k+1}},\dots ,\mathbf{i}_{C_{n}}}\left\Vert U\left(
x_{i_{1}}^{(1)},\dots ,x_{i_{m}}^{(m)}\right) \right\Vert _{Y}^{q}\right) ^{%
\frac{r_{k}}{q}}\right) ^{\frac{q}{r_{k}}}\right) ^{\frac{1}{q}}\right]
^{\theta _{k}} \\
&\leq \prod_{k=1}^{n}\left[ \mathrm{A}_{q,r_{k}}^{\left\vert \widehat{C_{k}}%
\right\vert }(Y)\left\Vert U^{C_{k}}:X^{\widehat{C_{k}}}\rightarrow \Pi
_{(r_{k},1)}^{\left\vert C_{k}\right\vert }\left( X^{C_{k}};Y\right)
\right\Vert \right] ^{\theta _{k}}.
\end{align*}
\end{proof}

\bigskip

As a particular case of this result, we obtain one of the main results of 
\cite{dps.coord}. Before, we need to recall some technical definitions (see 
\cite[Section 3]{dps.coord}): for $q\geq 2$, let us consider the functions $%
w,f: [1,q)^2\to [0,+\infty)$ defined by 
\begin{equation*}
\omega(x,y) := \frac{q^2\left( x+y\right) -2qxy}{q^{2}-xy} \ \ \ \text{ and }
\ \ \ f(x,y):= \frac{q^2 x-qxy}{q^2\left(x+y\right)-2qxy}.
\end{equation*}
Inductively, one may define $w_n : [1,q)^n\to [0,+\infty)$ by $%
\omega_2\left(x_1,x_2\right) := \omega \left(x_1,x_2\right)$, and, for $n
\geq 3$, 
\begin{equation*}
\omega_n \left(x_1,\dots,x_n\right) := \omega_2 \left(
x_n,\omega_{n-1}\left(x_1,\dots,x_{n-1}\right) \right).
\end{equation*}
We proceed similarly for $f_n := \left(f_n^1,\dots,f_n^n\right) : [1,q)^n
\to [0,+\infty)^n$. First, $f_{2}\left( x_{1},x_{2}\right) :=\left( f\left(
x_{1},x_{2}\right) ,f\left( x_{2},x_{1}\right) \right). $ Inductively, the
function $f_n$ (in $n$ variables $x_1,\dots,x_n$) is defined using $f_{n-1}$
(in the $n-1$ variables $x_1,\dots,x_{n-1}$) by 
\begin{equation*}
f_n^k (x_1,\dots,x_n) := f_{n-1}^k (x_1,\dots,x_{n-1}) \cdot f \left(
\omega_{n-1} \left(x_1,\dots,x_{n-1}\right),x_n \right),\ k=1,\dots n-1,
\end{equation*}
and 
\begin{equation*}
f_n^n (x_1,\dots,x_n) := f \left(
x_n,\omega_{n-1}\left(x_1,\dots,x_{n-1}\right) \right).
\end{equation*}
For any choice of $\left(x_1,\dots,x_n\right) \in [1,q)^n$, it can be
checked by induction that 
\begin{equation*}
\sum_{k=1}^n f_n^k\left(x_1,\dots,x_n\right) =1.
\end{equation*}

\vskip.3cm

Now let us see how to recover the main result of \cite{dps.coord} from
theorem \ref{4.2}.

\begin{corollary}
\label{4.2D} Let $\{1,\dots ,m\}$ be the disjoint union of non-void proper
subsets $C_{1},\dots ,C_{n}$, let $Y$ be a Banach space with cotype $q$, and
suppose that $1\leq r_{1},\dots ,r_{n}<q$. Assume that $U\in \mathcal{L}
\left( X_{1},\dots ,X_{m};Y\right) $ is multiple $\left( r_{k},1\right) $%
-summing in each $C_{k},\,1\leq k\leq n$. Then $U$ is multiple $\left(\omega
_{n},1\right) $-summing, and 
\begin{equation*}
\pi _{\left( \omega _{n},1\right) }^{m}\left( U\right) \leq \sigma_{n}
\prod\limits_{k=1}^{n} \left\Vert U^{C_{k}}:X^{\widehat{C_{k}}}\to \Pi
_{r_{k},1}^{|C_k|}\left( X^{C_{k}};Y\right) \right\Vert^{f_n^k},
\end{equation*}
where $\sigma_n$ is defined by 
\begin{equation*}
\sigma_{2}=\left( \mathrm{A}_{q,r_{1}}^{\left\vert C_{2}\right\vert }\left(
Y\right) \right) ^{f\left( r_{1},r_{2}\right) }\left( \mathrm{A}%
_{q,r_{2}}^{\left\vert C_{1}\right\vert }\left( Y\right) \right) ^{f\left(
r_{2},r_{1}\right) },
\end{equation*}
and for $n\geq 3$ 
\begin{equation*}
\sigma _{n}=\left( \mathrm{A}_{q,r_{n}}^{\left\vert \cup
_{k=1}^{n-1}C_{k}\right\vert }\left( Y\right) \right) ^{f\left( r_{n},\omega
_{n-1}\right) }\left( \mathrm{A}_{q,\omega _{n-1}}^{\left\vert
C_{n}\right\vert }\left( Y\right) \right) ^{f\left( \omega
_{n-1},r_{n}\right) }\sigma _{n-1}^{f\left( \omega _{n-1},r_{n}\right) }.
\end{equation*}
\end{corollary}

\begin{proof}
Using the following formulas for the exponents $\omega_{n}:=\omega_n%
\left(r_1,\dots,r_n\right)$ and $f_n^k:= f_n^k\left(r_1,\dots,r_n\right)$
(see \cite[Theorem 3.2]{ps}) 
\begin{equation*}
\omega_n = \frac{qR}{1+R} \ \ \ \text{ and } \ \ \ f_n^k= \frac{r_k}{R(q-r_k)%
},\, k=1,\dots,n, \ \ \ \text{where} \ \ \ R:=\sum_{k=1}^{n}\frac{r_k}{q-r_k}%
,
\end{equation*}
and taking $\theta_k := f_n^k,\, k=1,\dots,n$, in theorem \ref{4.2}, we get 
\begin{equation*}
\frac{1}{q_k} = \frac{1}{R(q-r_k)} + \frac{1}{q}\left(1-\frac{r_k}{R(q-r_k)}%
\right) = \frac{1+R}{qR} = \frac{1}{\omega_n},
\end{equation*}
for $k=1,\dots,n$. Thus theorem \ref{4.2} guarantees that $U$ is $%
(\omega_n,1)$-summing and 
\begin{equation*}
\pi_{(\omega_n,1)}^{m}(U) \leq \prod_{k=1}^n \left[ \mathrm{A}_{q,r_k}^{|%
\widehat{C_k}|}(Y) \right]^{f_n^k} \cdot \left[ \left\Vert U^{C_k} : X^{%
\widehat{C_k}} \to \Pi_{r_k,1}^{|C_k|} \left(X^{C_k};Y\right) \right\Vert %
\right]^{f_n^k}.
\end{equation*}
This is precisely the result stated, up to the constants $\sigma_n$ for $%
n\geq3$. In order to recover these, one need to proceed by induction as
described in the proof of \cite[Theorem 4.2]{dps.coord}, using that $U$ is
multiple $(\omega_{n-1},1)$-summing in the coordinates of $%
\cup_{k=1}^{n-1}C_k$, and by assumption that $U$ also it is multiple $%
(r_n,1) $-summing in the coordinates of $C_n$.
\end{proof}

The following important special case is highlighted in \cite[Section 3]%
{dps.coord} as an immediate consequence of the previous result.

\begin{corollary}[{\protect\cite[Section 3]{dps.coord}; Corollary 5.2}]
\label{5.2}Let $Y$ be a Banach space with cotype $q$, and $1\leq r<q$. Then
there is a constant $\sigma _{m}\geq1$ such that each separately $\left(
r,1\right) $-summing operator $U\in\mathcal{L}\left(
X_{1},...,X_{m};Y\right) $ is multiple $\left( \frac{qrm}{q+\left(
m-1\right) r},1\right) $-summing, and%
\begin{equation*}
\pi_{\left( \frac{qrm}{q+\left( m-1\right) r},1\right) }^{m}\left( U\right)
\leq\sigma_{m}\prod \limits_{k=1}^{m}\left\Vert U^{\{k\}}:X^{\widehat{\{k\}}%
}\rightarrow\Pi_{r,1}\left( X^{\{k\}};Y\right) \right\Vert ^{\frac{1}{m}}
\end{equation*}
where $\sigma_{m}$, as stated in Corollary \ref{4.2D}, depends on $m,~r,~q$
and $C_{q}\left( Y\right) $.
\end{corollary}

\bigskip

In the next section, we show that the previous result is a particular case
of an even more general theorem.

\section{Multiple $N$-separate summability}

The following definition is a variation of Definition \ref{one}.

\begin{definition}
Let $1\leq r<\infty $. We say that $U\in \mathcal{L} \left(X_1,\dots,X_m;Y%
\right)$ is \emph{$N$-separately $(r,1)$-summing}, when $U$ is multiple $%
\left( r,1\right)$-summing in each subset of $\{1,\dots,m\}$ with
cardinality $N$.
\end{definition}

In other words, $U\in \mathcal{L}\left( X_{1},...,X_{m};Y\right) $ is $N$%
-separately $\left( r,1\right) $-summing if $U$ is multiple $\left(
r,1\right) $-summing in $S\subseteq \{1,\dots ,m\}$, for all $S\in \mathcal{P%
}_{N}(m)$. In this context, $U$ is separately $(r,1)$-summing if and only if 
$U$ is $1$-separately $(r,1)$-summing.

From now on $Y$ is a Banach space with cotype $q$. The following result
extends Corollary \ref{5.2}:

\begin{theorem}
\label{ANSS} Let $1\leq r\leq q$, and $1\leq n<m$. If $U\in \mathcal{L}%
\left( X_{1},...,X_{m};Y\right) $ is $n$-separately $\left( r,1\right) $%
-summing, then $U$ is $N$-separately $\left( r_{N},1\right) $-summing, for
all $n<N\leq m,$ with $r_{N}:=\frac{qrN}{nq+\left( N-n\right) r}$. Moreover,
if $N<m$, we get, for each $D\in \mathcal{P}_{N}\left( m\right) $, 
\begin{equation*}
\pi _{\left( r_{N},1\right) }^{N}\left( U_{x_{\widehat{D}}}^{D}\right) \leq 
\mathrm{A}_{q,r}^{N-n}\left( Y\right) \prod_{S\in \mathcal{P}%
_{n}\left(D\right)} \left\Vert \left( U_{x_{\widehat{D}}}^{D}\right)
^{S}:X^{D\setminus S} \to \Pi _{(r,1)}^{n}\left( X^{S};Y\right) \right\Vert^{%
\frac{1}{\binom{N}{n}}}
\end{equation*}%
for all $x_{\widehat{D}}\in X^{\widehat{D}}$. The estimate for $N=m$ becomes 
\begin{equation*}
\pi _{\left( r_{m},1\right) }^{m}\left( U\right) \leq \mathrm{A}%
_{q,r}^{m-n}\left( Y\right) \prod_{S\in \mathcal{P}_{n}\left( m\right)
}\left\Vert U^{S}:X^{\widehat{S}}\to \Pi _{(r,1)}^{n}\left( X^{S};Y\right)
\right\Vert ^{\frac{1}{\binom{m}{n}}}.
\end{equation*}
\end{theorem}

\begin{proof}
Firstly, we will prove the result for $n<N<m$. Let $D\in \mathcal{P}%
_{N}\left( m\right)$. Without loss of generality, we may assume that $%
D=\{1,\ldots,N\}$. We must prove that $U^{D}$ has its range in $\Pi _{\left(
r_{N},1\right) }^{\left\vert D\right\vert }\left( X^{D};Y\right)$, that is,
given $x_{\widehat{D}}\in X^{\widehat{D}}$, $U^{D}_{x_{\widehat{D}}} \in \Pi
_{\left( r_{N},1\right) }^{\left\vert D\right\vert }\left( X^{D};Y\right) $.
Clearly, $U^{D}_{x_{\widehat{D}}}$ is bounded and $N$-linear. On the other
hand, since $U$ is $n$-separately $(r,1)$-summing, $U^{D}_{x_{\widehat{D}}}$
is $n$-separately $(r,1)$-summing, \emph{i.e.}, $U^{D}_{x_{\widehat{D}}}$ is 
$(r,1)$-summing in $S\subset D$, for all $S\in \mathcal{P}_{n}(D) = \mathcal{%
P}_{n}(N)$. Let $M$ be a positive integer and $x_{1}^{(k)},%
\ldots,x_{M}^{(k)} \in X_k$ be such that $\left\Vert
\left(x_{j}^{(k)}\right)_{j=1}^{M} \right\Vert_{w,1}\leq 1$, for $%
k=1,\dots,N $. Also, let us set $x_\mathbf{i} :=
\left(x_{i_1}^{(1)},\ldots,x_{i_N}^{(N)}\right) \in X^D$, for $\mathbf{i}
=\left(i_1,\dots,i_N\right) \in \mathcal{M}(N,M) = \{1,\ldots,M\}^N$. Lemma %
\ref{blei.interp} implies 
\begin{equation*}
\left( \sum_{\mathbf{i}} \left\Vert U^{D}_{x_{\widehat{D}}}\left(x_{\mathbf{i%
}}\right) \right\Vert^{r_{N}}\right)^{\frac{1}{r_{N}}} \leq \prod_{S\in 
\mathcal{P}_{n}(N)} \left( \sum_{\mathbf{i}_{S}} \left( \sum_{\mathbf{i}%
_{D\setminus S}} \left\Vert U^{D}_{x_{\widehat{D}}} \left(x_{\mathbf{i}%
}\right)\right\Vert^{q} \right)^{\frac{r}{q}} \right)^{\frac{1}{r}\cdot 
\frac{1}{\binom{N}{n}}},
\end{equation*}
with the sum $\sum_{\mathbf{i}}$ taken over all multi-index $\mathbf{i}%
=\left(i_1,\ldots,i_N\right) \in \mathcal{M}(N,M)$. Since $U^{D}_{x_{%
\widehat{D}}}$ is $n$-separately $(r,1)$-summing, Theorem \ref{DPS} assures
that 
\begin{equation*}
\left( \sum_{\mathbf{i}_{S}}\left( \sum_{\mathbf{i}_{D\setminus S}}
\left\Vert U^{D}_{x_{\widehat{D}}} \left(x_{\mathbf{i}}\right) \right\Vert
^{q}\right) ^{\frac{r}{q}}\right) ^{\frac{1}{r}} \leq \mathrm{A}%
_{q,r}^{N-n}\left( Y\right) \left\Vert \left( U^{D}_{x_{\widehat{D}}}
\right) ^{S}:X^{D\setminus S}\rightarrow \Pi _{(r,1)}^{n}\left(
X^{S};Y\right) \right\Vert .
\end{equation*}
Therefore, 
\begin{equation*}
\left( \sum_{\mathbf{i}} \left\Vert U^{D}_{x_{\widehat{D}}}\left(x_{\mathbf{i%
}}\right) \right\Vert^{r_{N}}\right)^{\frac{1}{r_{N}}} \leq \mathrm{A}%
_{q,r}^{N-n}(Y) \prod_{S\in \mathcal{P}_{n}(N)} \left\Vert \left( U^{D}_{x_{%
\widehat{D}}}\right)^S: X^{D\setminus S}\rightarrow \Pi _{(r,1)}^{n}
\left(X^{S};Y\right) \right\Vert^\frac{1}{\binom{N}{n}},
\end{equation*}
and this conclude the result for $n<N<m$. For $N=m$, one just need to work
with the maps $U^{S}: X^{\widehat{S}} \rightarrow \Pi _{(r,1)}^{n}
\left(X^{S};Y\right)$, for each $S\in \mathcal{P}_{n}(m)$, and follows the
lines of the previous argument.
\end{proof}

Notice that if $U$ is $1$-separately $\left(r,1\right)$-summing, then it is $%
N$-separately $\left( \frac{qrN}{q+\left( N-1\right) r},1\right)$-summing
for all $1\leq N\leq m$. To recover Corollary \ref{5.2}, i.e., $U\in 
\mathcal{L}\left(X_{1},\ldots,X_{m};Y\right)$ is multiple $\left( \frac{qrm}{%
q+\left(m-1\right) r},1\right)$-summing, $U$ just need to be $n$-separately $%
\left(s,1\right)$-summing for some $1\leq n<m$ and $s \leq \frac{qrn}{%
q+\left(n-1\right)r}$.

\vskip.1cm

We observe that, in some special cases, our approach provides better
exponents. In fact, let $1<n<N\leq m$ and suppose that $U$ is $n$-separately 
$\left( r,1\right) $-summing. Let $k,l\in \mathbb{N}$, with $l<n,$ be such
that $N=kn+l$. Thus, if $l\neq 0$, given $S\in \mathcal{P}_{N}(m)$, we may
choose $C_{1},\dots,C_{k}\in \mathcal{P}_{n}(m)$ and $C_{k+1}\in \mathcal{P}%
_{l}(m)$ such that 
\begin{equation*}
\bigcup_{j=1}^{k+1}C_{j}=S
\end{equation*}
with this union be disjoint. Clearly, since $l<n$ we conclude that $U$ is
multiple $\left( r,1\right) $-summing in the coordinates of $C_{k+1}$ and
(using the hypothesis) $U$ is multiple $\left( r,1\right) $-summing in the
coordinates of $C_{j}$ for $1\leq j\leq k$. Since $\omega _{k+1}\left(
r,\dots ,r\right) =\frac{q\left( k+1\right) r}{q+kr}$, using \cite[Theorem
5.1]{dps.coord} and the arbitrariness of $S\in \mathcal{P}_{N}(m)$, one may
conclude that $U$ is $N$-separately $\left(\frac{q\left(k+1\right) r}{q+kr}%
,1\right)$-summing. Nevertheless, Theorem \ref{ANSS} assures that $U$ is $N$%
-separately $\left(\frac{qrN}{nq+(N-n)r},1\right)$-summing. Note that, since 
$l \neq 0$, 
\begin{equation*}
\frac{q\left( k+1\right) r}{q+kr}>\frac{qrN}{nq+\left( N-n\right) r}.
\end{equation*}
If $l=0$, we will obtain that $\omega _{k}\left( r,\dots ,r\right) =\frac{qrk%
}{q+\left( k-1\right) r}=\frac{qrN}{nq+\left( N-n\right) r}$. Therefore, the
exponent provided by Theorem \ref{ANSS} is more efficient.

\vskip.1cm

As a final remark we note that Theorem \ref{ANSS} is also useful to provide
estimates for the constants involved. For instance, if we take $%
X_{1}=\cdots=X_{m}=c_{0}$ and $Y=\mathbb{K}$, we obtain better estimates to
the constants of some variation of Bohnenblust-Hille inequalities introduced
in \cite[Appendix A]{nps} and \cite{npss}. More precisely, it shows that for
all parameters $1 \leq t<2$ and all $m\in \mathbb{N}$, there exists a
constant $C_{m,t}^{\mathbb{K}}\geq 1$, such that, 
\begin{equation*}
\left( \sum_{i_{1},...,i_{m}=1}^{\infty }\left\vert U\left(
e_{i_{1}},...,e_{i_{m}}\right) \right\vert ^{\frac{2tm}{2+\left( m-1\right) t%
}}\right) ^{\frac{2+\left( m-1\right) t}{2tm}}\leq C_{m,t}^{\mathbb{K}%
}\left\Vert U\right\Vert ,
\end{equation*}
for all $m$-linear forms $U:c_{0}\times \cdot \cdot \cdot \times
c_{0}\rightarrow \mathbb{K\,\ }$, with 
\begin{equation}
C_{m,t}^{\mathbb{K}}=\left\{ 
\begin{array}{ll}
1 & \text{ if }m=1, \\ 
\left( \mathrm{A}_{\frac{2mt}{\left( m-2\right) t+4}}^{\mathbb{K}}\right) ^{%
\frac{m}{2}}C_{\frac{m}{2},t}^{\mathbb{K}} & \text{ if }m\text{ is even, and}
\\ 
\left( \left( \mathrm{A}_{\frac{2\left( m-1\right) t}{\left( m-3\right) t+4}%
}^{\mathbb{K}}\right) ^{\frac{m+1}{2}}C_{\frac{m-1}{2},t}^{\mathbb{K}
}\right) ^{\frac{m-1}{2m}}\left( \left( \mathrm{A}_{\frac{2\left( m+1\right)
t}{\left( m-1\right) t+4}}^{\mathbb{K}}\right) ^{\frac{m-1}{2}}C_{\frac{m+1}{%
2},t}^{\mathbb{K}}\right) ^{\frac{m+1}{2m}} & \text{ if }m\text{ is odd.}%
\end{array}
\right.  \label{vari}
\end{equation}

This can be easily inserted in the context of multiple multilinear forms:
for each parameter $t \in [1,2)$, we have a coincidence result for $m$%
-linear maps 
\begin{equation*}
\mathcal{L}\left(c_0,\ldots,c_0;\mathbb{K}\right) = \Pi_{\left(\frac{2tm}{%
2+(m-1)t},1\right)}^{m} \left( c_0,\ldots,c_0;\mathbb{K} \right),
\end{equation*}
which means that every bounded $m$-linear forms $U:c_0 \times \cdots \times
c_0 \to \mathbb{K}$ is multiple $\left( \frac{2tm}{2+\left(m-1\right)t}%
,1\right)$-summing. Moreover, the following norm estimates holds: 
\begin{equation*}
\pi_{\left(\frac{2tm}{2+(m-1)t},1\right)}^{m} (U) \leq C_{m,t}^{\mathbb{K}%
}\left\Vert U\right\Vert.
\end{equation*}

Combining this with Theorem \ref{ANSS}, the following estimates for the
variants of Bohnenblust-Hille inequality arises.

\begin{theorem}
Let $1\leq t<2$ and $m>1$. Then 
\begin{equation*}
\left( \sum_{i_{1},\dots ,i_{m}=1}^{\infty } \left\vert
U\left(e_{i_{1}},\dots ,e_{i_{m}}\right) \right\vert ^{\frac{2tm}{2+(m-1)t}}
\right)^{\frac{2+(m-1)t}{2tm}} \leq C_{m,t}^{\mathbb{K}} \left\Vert
U\right\Vert,
\end{equation*}
for all bounded $m$-linear forms $U:c_{0}\times \cdots \times c_{0} \to 
\mathbb{K}$, with 
\begin{equation*}
C_{m,t}^{\mathbb{C}} \leq \prod_{j=2}^{m}\Gamma \left(2-\frac{2-t}{2+t(j-2)}%
\right)^{-\frac{2+t(j-2)}{2t(j-1)}},
\end{equation*}
and 
\begin{equation*}
C_{m,t}^{\mathbb{R}} \leq 
\begin{cases}
2^{\left(\frac1t-\frac12\right)\cdot\sum_{j=1}^{m-1}\frac1j} ,\,\text{ if }
m \leq \frac{2p_0+2t(1-p_0)}{t(2-p_0)}; \\ 
\left[ \prod_{j=2}^{m_{0}}2^{\frac{t+2m_{0}-2tm_{0}+mt+jtm_{0}-jmt-2}{%
2t\left( m_{0}-1\right) \left( j-1\right) }}\right] \cdot \left[
\prod_{j=m_{0}+1}^{m} \left( \frac{1}{\sqrt{\pi }}\Gamma \left( \frac{3}{2}- 
\frac{2-t}{2+t(j-2)}\right) \right)^{\frac{t(j-2)+2}{2t-2jt}}\right] ,\, 
\text{ if } m > \frac{2p_0+2t(1-p_0)}{t(2-p_0)};%
\end{cases}%
\end{equation*}
where $m_{0}$ is the largest integer not greater than $\frac{2p_0+2t(1-p_0)}{%
t(2-p_0)}$.
\end{theorem}

\begin{proof}
In our context we have that, for $t\in \lbrack 1,2)$ and $m\geq 1$, every
bounded $m$-linear forms $U:c_{0}\times \cdots \times c_{0}\rightarrow 
\mathbb{K}$ is $n$-separately $\left( \frac{2tn}{2+(n-1)t},1\right) $%
-summing, for all $1\leq n\leq m$. Thus, by considering $n=m-1$ and using
that $U\mathbb{\,}$ is $(m-1)$-separately $\left( \frac{2t(m-1)}{2+(m-2)t}%
,1\right) $-summing, we invoke Theorem \ref{ANSS} to conclude that $U$ is
multiple $\left( \frac{2tm}{2+(m-1)t},1\right) $-summing and 
\begin{equation*}
\pi _{\left( \frac{2tm}{2+\left( m-1\right) t},1\right) }^{m}\left( U\right)
\leq \mathrm{A}_{2,\frac{2t\left( m-1\right) }{2+\left( m-2\right) t}}\left( 
\mathbb{K}\right) \prod_{S\in P_{m-1}\left( m\right) }\left\Vert U^{S}:X^{%
\widehat{S}}\rightarrow \Pi _{\left( \frac{2t\left( m-1\right) }{2+\left(
m-2\right) t},1\right) }^{m-1}\left( X^{S};\mathbb{K}\right) \right\Vert ^{%
\frac{1}{\binom{m}{m-1}}}.
\end{equation*}%
Since for $Y=\mathbb{K}$, we can use $\mathrm{A}_{\frac{2t(m-1)}{2+(m-2)t}}^{%
\mathbb{K}}$ instead of $\mathrm{A}_{2,\frac{2t(m-1)}{2+(m-2)t}}(\mathbb{K}%
), $ and 
\begin{align*}
\left\Vert U^{S}:X^{\widehat{S}}\rightarrow \Pi _{\left( \frac{2t\left(
m-1\right) }{2+\left( m-2\right) t},1\right) }^{m-1}\left( X^{S};\mathbb{K}%
\right) \right\Vert & =\sup_{x\in B_{X^{\widehat{S}}}}\pi _{\left( \frac{%
2t\left( m-1\right) }{2+\left( m-2\right) t},1\right) }^{m-1}\left(
U_{x}^{S}\right) \\
& \leq C_{m-1,t}^{\mathbb{K}}\sup_{x\in B_{X^{\widehat{S}}}}\left\Vert
U_{x}^{S}\right\Vert \\
& \leq C_{m-1,t}^{\mathbb{K}}\left\Vert U\right\Vert ,
\end{align*}%
we get 
\begin{equation*}
\pi _{\left( \frac{2tm}{2+\left( m-1\right) t},1\right) }^{m}\left( U\right)
\leq A_{\frac{2t\left( m-1\right) }{2+\left( m-2\right) t}}^{\mathbb{K}%
}C_{m-1,t}^{\mathbb{K}}\left\Vert U\right\Vert .
\end{equation*}%
Thus, 
\begin{equation*}
C_{m,t}^{\mathbb{K}}\leq A_{\frac{2t\left( m-1\right) }{2+\left( m-2\right) t%
}}^{\mathbb{K}}C_{m-1,t}^{\mathbb{K}}.
\end{equation*}%
Proceeding by induction and using that $C_{1,t}^{\mathbb{K}}=1$, we obtain 
\begin{equation*}
C_{m,t}^{\mathbb{K}}=%
\begin{cases}
1,\,\text{ if }m=1; \\ 
\prod\limits_{k=1}^{m-1}\mathrm{A}_{\frac{2tk}{2+\left( k-1\right) t}}^{%
\mathbb{K}},\,\text{ if }m>1.%
\end{cases}%
\end{equation*}%
Finaly, using (\ref{kinreal}), (\ref{kinreal2}) and (\ref{kincomp}), we
obtain the result.
\end{proof}

By considering $t=1$, we recover the Bohnenblust-Hille constants presented
in \cite[Proposition 3.1]{bps} and a direct calculation shows that the above
theorem improves (\ref{vari}). Proceeding as in \cite[Corollary 3.2 and
Corollary 3.3]{bps}, we have an alternative formula that highlights the
asymptotic behavior of the constants.

\begin{theorem}
For any $t\in \lbrack 1,2)$, there exists $\kappa _{t,\mathbb{K}}>0$ such
that, for any $m\geq 1$, 
\begin{equation*}
C_{m,t}^{\mathbb{C}}\leq \kappa _{t,\mathbb{C}}m^{\frac{\left( \gamma
-1\right) \left( t-2\right) }{2t}},
\end{equation*}%
and 
\begin{equation*}
C_{m,t}^{\mathbb{R}}\leq \kappa _{t,\mathbb{R}}m^{\frac{\left( \gamma -2+\ln
2\right) \left( t-2\right) }{2t}}.
\end{equation*}
\end{theorem}

\end{document}